\documentclass[a4paper,12pt]{amsart}

\usepackage{verbatim}
\usepackage{amsfonts,amsmath,amssymb}
\usepackage{amssymb}

\mathchardef\emptyset="001F

\theoremstyle{plain}

\newtheorem{theorem}{Theorem}[section]
\newtheorem{lemma}[theorem]{Lemma}

\newtheorem{corollary}[theorem]{Corollary}
\newtheorem{remark}[theorem]{Remark}
\newtheorem{definition}[theorem]{Definition}

\theoremstyle{definition}
\theoremstyle{remark}
\numberwithin{equation}{section}
\newcommand{\dpr}{{\mathcal D}'}
\newcommand{\Div}{\mathrm{Div}}
\newcommand{\e}{\varepsilon}
\newcommand{\Om}{\Omega}
\newcommand{\weakst}{\stackrel{\ast}{\rightharpoonup}}
\newcommand{\weak}{\rightharpoonup}

\newcommand{\R}{{\mathbb R}}

\newcommand{\Mcal}{{\mathcal M}}

\newcommand{\Acal}{{\mathcal A}}
\newcommand{\Lcal}{{\mathcal L}}
\newcommand{\A}{{\mathbb A}}

\newcommand{\Int}{{\rm Int}}

\newcommand{\N}{{\mathbb N}}

\newcommand{\M}{{\mathbb M}}
\newcommand{\U}{{\mathcal U}}
\newcommand{\Mmn}{\M^{{m\times n}}}
\newcommand{\Mnn}{{\mathbb M^{n\times n}}}
\newcommand{\Mnns}{\M^{n\times n}_{\mathrm skw}}
\newcommand{\Mtten}{\M^{{2\times 10}}}

\newcommand{\Mtts}{\M^{3\times 3}_{\mathrm skw}}
\newcommand{\Mtwo}{\M^{2 \times 3}}
\renewcommand{\S}{{\mathbb S}}

\newcommand{\varro}{\varrho}

\newcommand{\no}{\noindent}
\newcommand{\non}{\nonumber}
\renewcommand{\div}{\hbox{{\rm div}}}

\newcommand{\supp}{\hbox{{\rm supp}}}
\newcommand{\Lin}{\hbox{{\rm Lin}}}
\newcommand{\dsp}{\displaystyle}
\newcommand{\rank}{\mathrm{rank}}
\newcommand{\dist}{\mathrm{dist}}

\newcommand{\curl}{{\rm curl}\,}

\newcommand{\Vd}{V_{\delta}}
\newcommand{\Gd}{G_{\delta}}
\renewcommand{\t}{\theta}

\newcommand{\tP}{\widetilde P}
\usepackage{color}

\title[]{On a differential inclusion related to the Born-Infeld equations}

\author[Stefan M\"uller]{Stefan M\"uller$^{1}$}
\author[Mariapia Palombaro]{Mariapia Palombaro$^{2}$}

\begin{document}
\baselineskip3.15ex
\vskip .3truecm

\maketitle
\small
\noindent
$^1$ Hausdorff Center for Mathematics $\&$ Institute for Applied Mathematics, Universit\"at Bonn, Endenicher Allee 60, 53115 Bonn, Germany.
Email: sm@hcm.uni-bonn.de\\
\noindent
$^2$ Department of Information Engineering, Computer Science and Mathematics, 
University of L'Aquila, Via Vetoio 1, 67100 L'Aquila, Italy.
E-mail:  mariapia.palombaro@univaq.it

\vspace{2mm}

\begin{abstract}
We study a partial differential relation that arises in the context of the Born-Infeld equations
(an extension of the Maxwell's equations) by using Gromov's method of 
convex integration in the setting of divergence free fields.
\end{abstract}

\section{Introduction}

Let $\Om\subset\R^n$ be an open bounded set and let $K\subset\Mmn$ be a set of 
$m\times n$ real matrices. We study the problem of whether there exist solutions to the 
differential inclusion

\begin{equation}\label{general-inclusion}
\begin{cases}
\Div V=0 &\,\,\text{ in }\dpr(\Om;\R^m)\,,\\
V\in K  &\,\,\text{ a.e. in }\Om \,,\\
 -\hspace{-3.5mm}\int_{\Om} V= F \,,
\end{cases}
\end{equation}
 
\no
for some given $F\in\Mmn$.  
Our interest in this question arises, in particular, from applications to the study of the Born-Infeld equations. 
In fact, we will consider a special case of \eqref{general-inclusion}, when $m=2$, $n=3$, and the set $K$ 
is related to the so-called Born-Infeld manifold.
Further applications of solenoidal differential inclusions can be found in the study of composite materials, 
as well as linear elasticity and fluid mechanics 
(see, e.g., \cite{garroni-nesi}, \cite{palombaro-smysh}, \cite{delellis-szek}, \cite{delellis-szek2 }). 
More generally, problem \eqref{general-inclusion} falls into the framework of $\Acal$-quasiconvexity, 
where the differential constraint on the function $V$ is replaced by more general ones 
(see, e.g., \cite{fon-muller} and \cite{p} for related issues).

\par
Our approach to \eqref{general-inclusion} is based on studying the method of convex integration  
in the  div-free setting.
Convex integration  has been introduced and developed by Gromov to solve partial differential 
relations, in particular  in connection with geometric problems. 
An important problem is to find gradient fields  that take values in a prescribed set of matrices. This can be
written as the partial differential relation 
\begin{equation}  \label{eq:pdr_gradient}
\nabla u\in K\,.
\end{equation}

\no
We refer to Gromov's treatise \cite{gromov}  for a detailed exposition and further references concerning 
the existence of $C^1$ solutions. Gromov only very briefly discusses the existence of Lipschitz solutions
(see \cite{gromov}, p. 218)
and a more detailed  theory of Lipschitz solutions has been developed in a number of contributions including
\cite{muller-sverak,dacorogna_marcellini_acta,dacorogna_marcellini_book,sychev,kirchheim_lecture_notes, kirc-muller-sverak}
and has lead to a number of  new results, e.g.,  in the study of solid-solid phase-transitions, counterexamples to the regularity
of elliptic systems \cite{muller-sver03, szekelyhidi_regularity} and mathematical origami \cite{dacorogna_origami}.
The  partial differential relation \eqref{eq:pdr_gradient} corresponds (locally) to the constraint $\curl v = 0$.
In the spirit of Tartar's work \cite{tartar}  it is naturally to consider also constraints $ \Acal v = 0$ where $\Acal$ is a general
first order differential operator with constant coefficients. We deal with the case 
$\div V =0$ where $V$ is matrix-valued and the divergence is taken rowwise. The divergence constraint
has already been considered elsewhere  in the context of convex integration, e.g.,  in  
\cite{dacorogna-et-al}  in the context of general closed differential forms\footnote{The setting in \cite{dacorogna-et-al}
is both more general and more restrictive than our setting. First the authors consider the relation
$d \omega \in E$ where $\omega \in W^{1,\infty}(\Omega, \Lambda^k)$  is a general $k$-form on $\Omega \subset \R^n$ (our setting corresponds to 
$k=n-2$) and second they allow $E$ to be contained in a lower dimensional subspace. On the other 
hand their  treatment does not directly cover the case $\div V = 0$ if $V \in W^{1, \infty}(\Omega, \Mmn)$ and  $m \ge 2$.} and in  \cite{delellis-szek, delellis-szek2 } 
in the context of the Euler equations
(see also   \cite{szek3} and references therein). In Section \ref{conv-integ-section} we give a brief self-contained description of 
convex integration with the constraint $\div V=0$ since we will use exactly the same strategy for the application to the
Born-Infeld equation.


More precisely we show in Theorem \ref{mainthm} that problem \eqref{general-inclusion} 
admits a solution whenever $K$ can be ``approximated" in the sense of Definition \ref{in-approximation} 
and $F$ lies in the interior of some appropriate hull of $K$.

In Section \ref{young-meas} we specialize the results obtained in Section \ref{conv-integ-section} to the 
case of a partial differential relation arising in connection with the Born-Infeld equations. 
Let us briefly introduce the problem. 
The Born-Infeld system is a non-linear version of Maxwell's equations which can be written as
a set of partial differential constraints
\begin{align}\label{BI-system}
\partial_t D + \curl\Big(\dsp\frac{-B+D\wedge P}{h}\Big) & =
\partial_t B + \curl\Big(\dsp\frac{D+B\wedge P}{h}\Big) =0 \,,\\
\label{eq:initial_constraint}  \div D & = \div B=0\,,
\end{align} 
combined with the pointwise relation
\no

\begin{equation}\label{Ph}
P=D\wedge B,\qquad h=\sqrt{1+|B|^2 +|D|^2 +|P|^2} \,.
\end{equation}
Here $D,B, P:\Om\times[0,T]\subset\R^3\times\R^+\to\R^3$ and $h: \Om\times[0,T]\subset\R^3\times\R^+\to\R$.
Note that \eqref{BI-system} implies that $\partial_t \,  \div D = \partial_t  \, \div B = 0$. Thus if \eqref{eq:initial_constraint}
holds at time $t=0$ it holds for all times. 

\no
The relations \eqref{Ph} define a six-dimensional manifold  in  $\R^{10}$, that we call the BI-manifold and denote 
by $\Mcal$.
We refer to Brenier \cite{brenier} for the mathematical analysis and many further references on 
the Born-Infeld equations \eqref{BI-system}. Here we only give a brief account of those arguments of 
\cite{brenier} which give rise to the question addressed in this paper.
The starting point is to observe that if $(D,B)$ are smooth solutions of \eqref{BI-system} and if  $P$ and $h$
are given by \eqref{Ph} then  they
satisfy the additional conservation laws

\begin{align}\label{c.laws1}
\partial_t h + \div P & = 0 \,,\\
\label{c.laws2}  \partial_t P +\Div\Big(\frac{P\otimes P - B\otimes B- D\otimes D}{h}  \Big) & 
= \curl\Big(\frac{1}{h}\Big)\,.
\end{align}

\no
This suggests to lift the $6\times 6$ system \eqref{BI-system} 
to a $10\times 10$ system of conservation laws by adding equations \eqref{c.laws1}, \eqref{c.laws2}, regardless of 
the condition \eqref{Ph}.
More precisely one regards $P$ and $h$ as additional unknowns and considers the  
augmented system  \eqref{BI-system}, \eqref{c.laws1}, \eqref{c.laws2}. This system  enjoys remarkable 
properties which allow for an easier analysis than the original system \eqref{BI-system} 
(see \cite{brenier} for more precise details). 
Of course, among all solutions of the augmented system, only those with 
initial conditions valued in the BI-manifold genuinely correspond to the original system \eqref{BI-system}.
A natural question  is  which initial conditions can be weakly approximated 
by initial conditions valued in the BI-manifold $\Mcal$.
Brenier shows that the convex hull of the six-dimensional set $\Mcal$ contains an open
set in $\R^{10}$ and then states without proof: 'From this result, we infer that, through weak completion,
we may consider, for the ABI system, all kinds of initial condition with full dimensionality, where the
'fluid variables' $(h,P)$ are clearly distinct from the 'electromagnetic' variables.' (\cite{brenier}, p. 73). 

Here we provide a proof of the statement that all vectors $F$  in the interior of the convex
hull  of $\Mcal$ can arise as initial conditions 'through completion'. A soft version of this statement 
is that given $F\in\Int(\Mcal^c)$ there exists 
$\dsp \{V_{j}\}=\{(D_{j},B_{j},P_{j},h_{j})\}\subset L^{p}(\Om;\R^{10})$
such that 
\begin{equation}\label{weakpb1}
V_j \weak F  \quad \mbox{weakly in $L^p$}
\end{equation}
and
\begin{equation}\label{weakpb2}
\div D_{j}=\div B_{j}=0 \quad\text{and }\quad \dist(V_{j},\Mcal)\to 0 \quad\text{in measure}\,,
\end{equation}

\no
or 

\begin{equation*}
\div D_{j}\to0\,, \div B_{j}\to 0 \quad\text{strongly in }W^{-1,p'}\quad \text{and}
\quad V_{j}\in\Mcal \text{ a.e}\,.
\end{equation*}
%
%
%
%
%
%
This is proved in Section \ref{young-meas} and is essentially a consequence
of Tartar's approach, see  \cite{tartar} and \cite{fon-muller}

In Section \ref{section-A} we prove the following stronger result which
shows that there exist an approximating sequence which satisfies
both constraints $\div B_j = \div D_j=0$ and $V_j \in \Mcal$ exactly.
Since it requires almost no extra work we allow piecewise constant  functions $F$ rather than just
constant $F$. 

Here and in the following we say that a function $f: \Omega \to \R^m$ is {\it piecewise
constant} if there exist (finitely or countably many) mutually disjoint open sets $\Omega_i$ with Lipschitz boundary such that
\begin{equation}
\mbox{$f_{| \Omega_i}$ is constant} \quad \mbox{and} \quad   |\Omega \setminus \bigcup_i \Omega_i| = 0,
\end{equation}
 where $|E|$ denotes the Lebesgue measure of $E$. Similarly we say that $f$ is {\it piecewise affine} 
 if there exists $\Omega_i$ as above and 
 \begin{equation}
\mbox{$f_{| \Omega_i}$ is affine}. 
\end{equation}
The assumption that $\Omega_i$ should have Lipschitz boundary is natural for piecewise affine functions. 
It can actually be dropped by showing the perturbations we use are always in $W^{1, \infty}_0(\Omega_i)$.
For this we only need that the explicit diamond shaped set $\tilde \Omega_\e$ which is defined in the proof of
Lemma \ref{basic} has Lipschitz boundary.


%
%
%

\begin{theorem}\label{final-thm} Let $\Omega \subset \R^3$ be a bounded and  open set with Lipschitz boundary. 
Let $L$ be a compact subset of $ \Int(\Mcal^c)$, let $F \in L^\infty(\Om;\R^{10})$ and suppose
that $F$ is piecewise constant and satisfies
\begin{equation}
F(x) \in  L  \quad \mbox{a.e.}
\end{equation}
as well as 
\begin{equation}
\div D =\div B=0  \quad\text{in }\dpr(\Om).
\end{equation}
Then there exists a sequence 
$\dsp \{V_{j}\}=\{(D_{j},B_{j},P_{j},h_{j})\}\subset L^{\infty}(\Om;\R^{10})$ such that 
\begin{align*}
& \div D_{j}=\div B_{j}=0  \quad\text{in }\dpr(\Om),\\
& V_{j}\in\Mcal \quad\text{a.e.}\,,\\
& V_j \weakst F \quad\text{in }L^{\infty}-\text{weak*}\,.
\end{align*}
\end{theorem}

%

\no
Of course Theorem \ref{final-thm} is only useful if the convex hull of the set 
$\Mcal$ has non-empty interior. This follows from the following result of Brenier.

\begin{theorem}[\cite{brenier}, Thm. 2]     \label{convex-hull}
The convex hull $\Mcal^{c}$ satisfies
\begin{align}
\Mcal^{c}    & \supset \{
(B,D,P,h)\subset\R^{3}\times\R^{3}\times\R^{3}\times\R: h\geq 1+|D|+|B|+|P|
\}\, , \\
\Mcal^{c}  & \subset \{
(B,D,P,h)\subset\R^{3}\times\R^{3}\times\R^{3}\times\R: h\geq \sqrt{1+|D|^{2}+|B|^{2}+|P|^{2}}
\}\,.
\end{align}
\end{theorem}

\begin{proof} To keep this paper self-contained, we provide a short proof for the convenience of the reader.
The second inclusion is clear since the function
\begin{equation*}
f(B,D,P,h) := \sqrt{1+|D|^{2}+|B|^{2}+|P|^{2}} -h
\end{equation*}
is convex and vanishes in $\Mcal$. 

To prove the first inclusion it suffices to show that for every $s > 0$ the set
\begin{equation*}
B_s := \{(B,D,P,s)\subset\R^{3}\times\R^{3}\times\R^{3}\times\R: s\geq 1+|D|+|B|+|P|
\}
\end{equation*}
is contained in $\Mcal^c$. Now $B_s = \emptyset$ if $s < 1$ and $B_1 = \{1\} \times \{(0,0,0)\} \subset \Mcal$. 
For $s>1$ the set $B_s$ is convex and compact. 
We claim that its extreme points are given by
\begin{align*}
\mbox{Ext}\,(B_s) = \{ & (s,D,B,P) : 1 + |D| + |B| + |P| = s,  \\
& \mbox{only one of the vectors $D$, $B$, $P$ is non-zero} \}.
\end{align*}
Let $(D,B,P,s)$ be an extreme point of $B_s$.  Then $1+|D|+|B|+|P| = s$. 
Assume  $D \ne 0$ and $B \ne 0$. 
Then $(D + t D/|D|, B - t B/|B|, P,s) \in B_s$ for $|t| < \min(|B|, |D|)$ and thus $(D,B,P,s)$ is not an extreme point of $B_s$. 
Similarly one shows that  no other two vectors can be simultaneously non-zero.

Since $\mbox{Ext}\,(B_s)$ is compact we have $B_s = (\mbox{Ext}\,(B_s))^c$. It thus suffices to show that
$\mbox{Ext}\,(B_s) \subset \Mcal^c$. Consider a point of the form $Y = (D,0,0,s)$ with $|D| = s-1$. 
This point is a convex combination of $X_{\pm} := (D, \pm \alpha D, 0, s)$. We have $X_{\pm} \in \Mcal$
if and only if $1 + (1 + \alpha^2) |D|^2 = s^2$ and such an $\alpha$ exists since $1 + |D|^2 = s^2 - 2(s-1) < s^2$. 
Thus $Y\in \Mcal^c$. In the same way one shows that $(0, B,0, s) \in \Mcal^c$ if $|B| = s-1$. 
Finally consider $Y = (0,0,P,s)$ with $|P| = s-1$. There exist $d, b \in \R^3$ such that $(d,b,P/|P|)$ is a positively
oriented orthonormal basis. Let $D = \sqrt{s-1} \, d$, $B=\sqrt{s-1} \,  b  $. Then $d \wedge b = (s-1) p = P$
and $1 + |D|^2 + |B|^2 + |P|^2 = 1 + (s-1) + (s-1)  + (s-1)^2 =  s^2$. Hence $(D,B,P,s) \in \Mcal$ and
similarly $(-D,-B,P,s) \in \Mcal$. It follows that $(0,0,P,s) \in \Mcal^c$. 
\end{proof}

Serre  \cite{Serre} has shown  the sharper upper bound 
\begin{align}    
\Mcal^{c}  \subset \{
& (B,D,P,h)\subset\R^{3}\times\R^{3}\times\R^{3}\times (0, \infty) :  \nonumber \\
&  h^2 \geq 1+|D|^{2}+|B|^{2}+|P|^{2} +2|P-D\wedge B|
\} .
\end{align}
Very recently \cite{Serre2} he has improved the upper bound to 
\begin{align}    
\label{eq:serre_improved}
\Mcal^{c}  \subset \{
&
(B,D,P,h)\subset\R^{3}\times\R^{3}\times\R^{3}\times (0, \infty):   \nonumber \\
& h^2\geq 1+|D|^{2}+|B|^{2}+|P|^{2} +2\sqrt{|P-D\wedge B|^2
+ |P \cdot D|^2 + |P \cdot B|^2}
\} .
\end{align}
It seems natural to conjecture that equality holds in the last relation but this seems to be open.
The precise form of $\Mcal^c$ is not important for our argument.




\section{notation}

For a matrix $A=(A_{ij})\in\Mmn$ we denote by $A^{i}$ the $i$th column of $A$, and 
by $A_{i}$ the $i$th row of $A$.
We say that a matrix field $V\in L^{1}(\Om;\Mmn)$ is divergence free, 
and we write  $\Div V=0$ in $\dpr(\Om;\R^{m})$, if each row of the matrix field $V$ is 
divergence free in the distributional sense. 
We denote by $\Mcal$ the six-dimensional manifold in $\R^{10}$ defined as

\begin{equation}\label{BI-manifold}
\Mcal :=\{(D,B,P,h)\subset\R^{3}\times\R^{3}\times\R^{3}\times\R:\,
P=D\wedge B, h=\sqrt{1+|D|^{2}+|B|^{2}+|P|^{2}}\,\}\,,
\end{equation}
\no
and by $\Mcal^c$ its convex hull. For the topological interior of $\Mcal$ we 
write $\Int(\Mcal)$.
In Section \ref{section-A} we use the identification 
$\R^{10}\simeq\R^{3}_{D}\times\R^{3}_{B}\times\R^{3}_{P}\times\R_{h}$, and 
for any $M=(M_{1},\dots,M_{10})\in\R^{10}$, we write

\begin{equation*}
M=(M_{D},M_{B},M_{P},M_{h})\in
\R^{3}_{D}\times\R^{3}_{B}\times\R^{3}_{P}\times\R_{h}\,.
\end{equation*}

As usual $W^{1,\infty}(\Omega)$ denotes the Sobolev space of $L^\infty$ functions whose
distributional derivative are  in $L^\infty$. By $W^{1,\infty}_0(\Omega)$ we denote the subspace
of functions $f$ such that there exist $f_k \in C_c^\infty(\Omega)$ with
$(f_k , Df_k) \to (f, Df)$ a.e. and $\sup_k  \|f_k\|_{W^{1,\infty}} < \infty$.  If $\Omega$ is a bounded open set with Lipschitz
boundary then $W^{1, \infty}(\Omega)$ agrees with the space of functions which
have a Lipschitz continuous extension to $\bar \Omega$ and the  subspace $W^{1,\infty}_0(\Omega)$
consists exactly of Lipschitz functions with  $f_{| \partial \Omega} = 0$. 

If $f$ is a function on $E \subset \R^n$ we denote by $f \chi_E$ the extension of $f$ by zero to $\R^n$.  If $f \in W^{1, \infty}_0(\Omega)$ then approximation of $f$ by $f_k \in C_c^\infty(\Omega)$ shows that
\begin{equation}  \label{eq:extension}
f \chi_E  \in W^{1,\infty}(\R^n)  \quad \mbox{and}  \quad D (f \chi_E) = (Df) \chi_E  \mbox{   in  } \dpr(\R^n).
\end{equation}

\section{Convex integration for solenoidal fields}\label{conv-integ-section}

As already remarked in the introduction, extensions of the convex integration method 
to the div-free case are known. However, for the reader's convenience and because 
of certain modifications of the existing approaches, we present a self-contained program based on the notion of in-approximation. 
We will essentially follow  \cite{muller-sverak}. 


We will work with potentials of divergence free fields. Therefore we introduce 
the differential operator 
$\dsp
\Lcal : \big(W^{1,\infty}(\Om;\Mnn)\big)^m\to L^\infty(\Om;\Mmn)\,,
$
defined as

\begin{equation*}
\big(\Lcal (G)\big)_{kj} := \sum_{i=1}^{n} \frac{\partial G_{ij}^{k}}{\partial x_i}\,, 
\quad\,
1\leq k\leq m,\,  1\leq j\leq n\,, \quad G=(G^1,\dots,G^m)\,.
\end{equation*}

\begin{lemma}

Let $G^k\in W^{1,\infty}(\Om;\Mnn)$ be matrix fields for $1\leq k\leq m$,  such that 
the tensor $G^k$ is skew symmetric for every $k$, i.e., $\dsp G_{ij}^k=-G_{ji}^k$.  
Then the matrix field $\Lcal(G)$
is divergence free.
\end{lemma}

\begin{remark}
{\rm
For $n=3$, the space of skew symmetric $3 \times 3$ matrices $\Mtts$ can be identified with
$\R^3$ and the operator $\Lcal$ can alternatively be written as the rowwise curl of a $k \times 3$ matrix. 
We will, however, not use this fact. 
}
\end{remark}

The next result provides the basic construction that allows one to define a divergence free field   
whose values lie in a small neighborhood of two values, and whose potential can be chosen to be 
zero on the boundary. 

\begin{lemma}\label{basic}
Let $A,B\in\Mmn$ and let $F:=\t A+(1-\t)B$ for some $\t\in(0,1)$. Assume that $\rank(A-B)\leq n-1$.  
Then for each $\delta>0$, there exists $V\in L^{\infty}(\Om;\Mmn)$ such that 

\begin{align*}
& V=\Lcal (G) +F \text{ with } G\in \big(W^{1,\infty}(\Om;\Mnns)\big)^m \text { and piecewise linear}\,,\\
& \|G\|_{L^{\infty}(\Om)}<\delta   \,,\\
& G|_{\partial \Om}=0\,,\\
& \dist(V,\{A,B\})<\delta \,.
\end{align*}
\end{lemma}

\begin{proof}
Without loss of generality we may assume that $(A-B)e_{n}=0$, and $F=0$, 
so that we can write $A=(1-\t)(A-B)$ and $B=-\t(A-B)$. If not, we can replace $A$ and $B$ 
by $A-F$ and $B-F$ respectively.
We first construct a solution for a special domain $\Om_{\e}$ and then we will complete the proof by an application of the Vitali covering theorem. 
Let $\Om_{\e}:=(-1,1)^{n-1}\times (0,\e)$ and let 
$\chi :\Om_{\e}\to\{0,1\}$ be the characteristic function of the set 
$(-1,1)^{n-1}\times (0,\e\t)$: 

\begin{equation*}
\chi(x) = 
\begin{cases}
1    & \text{ if }\, 0\leq x_{n}\leq \e\t  \,,\\
0    & \text{ if }\, \e\t< x_{n}\leq \e\,. 
\end{cases}
\end{equation*} 

\no
We then define $U:=\chi A + (1-\chi) B$ and remark that $U$ is divergence free, since 
$(A-B)e_{n}=0$. 
We seek a potential $P$ of $U$. For each $k=1,\dots ,m$, and $j=1\dots,n$, let 

\begin{align*}
& P^k_{nj}(x) = 
\begin{cases}
A_{kj} x_{n}    & \text{ if }\,0\leq x_{n}\leq \e\t  \,,\\
B_{kj} (x_{n}-\e\t)+\e\t A_{kj}    & \text{ if }\,\e\t< x_{n}\leq \e\,,
\end{cases}\\
& P^k_{jn}=-P^k_{nj}  \,, \\
& P^k_{ij}=0 \:\text{ otherwise}.
\end{align*}

\no
It is readily seen that $\dsp U=\Lcal(P)$.
Moreover 
$P$ is piecewise linear and $P=0$ at $x_{n}=0$ and $x_{n}=\e$, but $P$ does not vanish on the whole 
boundary of $\Om_{\e}$. 
In order to find the sought function $G$, we first    
remark that, for each $k=1,\dots,m$, the function $P^k_n$ is proportional to $A_k-B_k$ and compute 
$\langle P^k_n,A_k-B_k\rangle$:

\begin{equation*}
\langle P^k_n,A_k-B_k\rangle = 
\begin{cases}
|A_k-B_k|^{2}(1-\t)x_{n}    & \text{ if }\,0\leq x_{n}\leq \e\t  \,,\\
|A_k-B_k|^{2}\t(\e -x_{n})   & \text{ if }\,\e\t< x_{n}\leq \e\,.
\end{cases}
\end{equation*} 

\no
Note that $\dsp \langle P^k_n,A_k-B_k\rangle \geq 0$ in $\Om_{\e}$. 
For each $k=1,\dots,m$, we introduce  the function 

$$
Q^k_n(x):=-\e\t(1-\t)(|x_{1}|+\dots +|x_{n-1}|)(A_k-B_k)
$$ 

\no
and set 
\begin{equation}\label{perturbedP}
\tP^k_n:=P^k_n+Q^k_n\,.
\end{equation}

\no 
The function $\tP^k_n$ is piecewise linear and satisfies  
$\dsp\langle \tP_n^k,A_k-B_k\rangle\leq 0$ on $\partial\Om_{\e}$. 
On the other hand $\dsp\langle \tP^k_n,A_k-B_k\rangle> 0$ in a 
neighborhood of the segment $\{x\in\Om_{\e}:x_{1}=\dots =x_{n-1}=0\}$. 
Set 

$$
\widetilde\Om_{\e}:=\{x\in\Om_{\e}: \langle \tP^k_n,A_k-B_k\rangle> 0\}\,,
$$

\no 
and define $\dsp \widetilde{U}:=\Lcal(\tP)$, where $\tP\in \big(W^{1,\infty}(\Om;\Mnn)\big)^m$ is defined 
by \eqref{perturbedP} and 

\begin{align*}
& \tP^k_{jn}=-\tP^k_{nj}  \,, \\
& \tP^k_{ij}=0 \:\text{ otherwise}.
\end{align*}

\no
Then 

\begin{align*}
& \tP\in \big(W^{1,\infty}(\widetilde\Om_{\e};\Mnn)\big)^m \text { is piecewise linear}\,,\\
& \tP|_{\partial \widetilde\Om_{\e}}=0\,,\\
& \|\tP\|_{L^{\infty}(\widetilde\Om_{\e})}< \e\t(1-\t)|A-B|\,,\\
& \dist(\widetilde{U},\{A,B\})< \e\t(1-\t)|A-B|\,.
\end{align*}

\no
By the Vitali covering theorem one can exhaust $\Om$ by disjoint scaled copies 
of $\widetilde\Om_{\e}$. More precisely, there exist $r_{i}\in (0,1)$ and 
$x_{i}\in \Om$ such that the sets 
$\widetilde\Om_{\e}^{i}:=x_{i}+r_{i}\widetilde\Om_{\e}$ are 
mutually disjoint, compactly contained in $\Omega$ and meas$(\Om\setminus\cup\widetilde\Om_{\e}^{i})=0$.
Then we define 

\begin{equation*}
G(x):= 
\begin{cases}
\dsp r_{i}\tP\big(r_{i}^{-1}(x-x_{i})\big)    & \text{ if } x\in \widetilde\Om_{\e}^{i} \,,\\
\dsp 0   & \text{ elsewhere }\,.
\end{cases}
\end{equation*} 
It follows from \eqref{eq:extension} that $G \in W^{1, \infty}(\Omega)$ and we have $G=0$ on $\partial \Omega$. 
We set $V:=\Lcal (G)$. By choosing $\e$ sufficiently small, it can be easily checked that $V$ satisfies all the required properties.

\end{proof}

Next we study the problem of finding a divergence free field taking values in an open set $K$ 
and with a prescribed average $F$. 
From Lemma \ref{basic} we know that such problem can be solved 
provided that $F=\t A+(1-\t)B$ for some $\t\in(0,1)$ and $A,B\in K$, with $\rank(A-B)\leq n-1$. 
We will see that this procedure can be iterated. More precisely, if $\rank(F-F')\leq n-1$, and 
$F'=\t' A'+(1-\t')B'$ for some $\t'\in(0,1)$ and $A',B'\in K$, with $\rank(A'-B')\leq n-1$, than the 
above problem can be solved also for $\mu F+(1-\mu)F'$ for all $\mu\in(0,1)$. 
This motivates the following definition.

\begin{definition}{\rm
We say that $K\subset\Mmn$ is stable under lamination (or {\em lamination convex}) 
if for all $A,B\in K$ such that $\rank(A-B)\leq n-1$, and all $\theta\in(0,1)$, one has
$\theta A +(1-\theta)B\in K$. 
The {\em lamination convex hull} $K^{L}$ is defined as the smallest lamination 
convex set that contains $K$.
}
\end{definition}

\begin{remark}\label{rem-K^{L}}
{\rm
It can be easily checked that the lamination convex hull $K^{L}$ is obtained by 
successively adding rank-$(n-1)$ segments, {\em i.e.}, 

\begin{equation*}
K^{L}=\bigcup_{i}K^{i}\,,
\end{equation*}

\no
where $K^{0}=K$ and 
$$ 
K^{i}:=K^{i-1}\cup\{C: \exists A,B\in K^{i-1}, \theta\in(0,1) \text{ \rm such that }
C=\theta A +(1-\theta)B, \rank(A-B)\leq n-1 \}\,.
$$
Moreover, if $K$ is open, than all the sets $K^{i}$ are open.
}
\end{remark}

\begin{lemma}\label{lemma-k-open}
Suppose that $K\subset\Mmn$ is open and bounded and that 
$F\in L^{\infty}(\Om;\Mmn)$  is  a piecewise constant function which satisfies 

\begin{align*}
& \Div F=0 \text { in }\dpr(\Om;\R^{m})\,,\\
& F\in K^{L} \text{ \rm a.e. }
\end{align*}

\no 
Then, for each $\delta>0$, there exists  $\Vd\in L^{\infty}(\Om;\Mmn)$ such that 
\begin{align*}
& \Vd=\Lcal (\Gd) +F \text{ with } \Gd\in \big(W^{1,\infty}(\Om;\Mnn)\big)^m
\text { and piecewise linear}\,,\\
& \Vd\in K \quad {\rm a.e.}\,, \\
& \|\Gd\|_{L^{\infty}(\Om)}<\delta   \,,\\
& \Gd|_{\partial \Om}=0\,.
\end{align*}
\end{lemma}

\begin{proof}
We first assume that $F$ is constant. Then $F\in K^{i}$ for some $i$. 
We argue by induction on $i$. 
If $i=1$, then the result holds by Lemma \ref{basic}. Now assume that the result 
is true for all $i\leq j$ and let $F\in K^{j+1}$. Then there exist $A,B\in K^{j}$ 
such that $\rank(A-B)\leq n-1$ and $F:=\t A+(1-\t)B$ for some $\t\in(0,1)$. 
By Lemma \ref{basic} there exists a piecewise linear function 
$G$ such that $\|G\|_{L^{\infty}(\Om)}<\delta/2$, $\dsp G|_{\partial \Om}=0$ and 
$\dsp \dist(\Lcal (G),\{A-F,B-F\})<\delta$. 
Since the set $K^{j}$ is open (see Remark \ref{rem-K^{L}}),
for sufficiently small $\delta$, the function 
$U:=\Lcal (G)+F$ satisfies $\dsp U\in K^{j}$ a.e. The latter inclusion implies that 
$U$ can be written in the form $\dsp U=\sum_{h}\chi_{_{\Om_{h}}}(C_{h}+F)$, with 
$C_{h}+F\in K^{j}$ and with $\chi_{_{\Om_{h}}}$ characteristic functions of disjoint 
open  subsets $\Omega_h$  of $\Om$ with Lipschitz boundary and  $|\Omega \setminus \bigcup_h \Omega_h| = 0$.
We can now apply the induction hypothesis on each subset $\Om_{h}$ to deduce the 
existence of functions $G_{h}\in \big(W^{1,\infty}(\Om_{h};\Mnn)\big)^m$ such that 
\begin{align*}
&\Lcal (G_{h})+C_{h}+F\in K  \text{ a.e. in }\Om_{h}\,, \\
&  \|G_{h}\|_{L^{\infty}(\Om_{h})}<\delta/2\,, \\ 
& G_{h}|_{\partial \Om_{h}}=0\,. 
\end{align*}

\no
Finally let $\dsp\Gd (x):=\sum_{h}\chi_{_{\Om_{h}}} G_{h}+G$. Then
$\dsp  \|\Gd\|_{L^{\infty}(\Om)}<\delta$ and $\dsp \Gd|_{\partial \Om}=0$ and
by \eqref{eq:extension} we have 
\begin{equation*}\Lcal(G_\delta) + F  = \sum_h\chi_{\Omega_h} (\Lcal(G_h) + C_h + F) \in K \mbox{    a.e.}
\end{equation*}

Now let $F$ be piecewise constant. Then $\dsp F=\sum_{k}\chi_{_{\Om_{k}}}F_{k}$ 
with $F_{k}\in K^{L}$. 
We now use the previous argument in each subdomain $\Om_{k}$ where $F$ is constant to obtain the existence of  piecewise linear functions 
$\dsp\Gd^{k}\in \big(W^{1,\infty}(\Om_{k};\Mnn)\big)^m$ such that 

\begin{align*}
& \Lcal(\Gd^{k})+F_{k}\in K \text{ \rm a.e.}\,,\\
&  \|\Gd^{k}\|_{L^{\infty}(\Om_{k})}<\delta   \,,\\
& \Gd^{k}|_{\partial \Om_{k}}=0\,.
\end{align*}

\no
Finally we define $\dsp \Gd :=\sum_{k}\chi_{_{\Om_{k}}} \Gd^{k}$ and set
$\dsp \Vd:= \Lcal(\Gd) + F$. Using again \eqref{eq:extension} we easily deduce the assertion.
\end{proof}

The next step is to pass from open sets to more general sets $K\subset\Mmn$. 
In order to do this we approximate 
$K$ by open sets $\U_i$ and we construct approximate solutions $V_i$ that  satisfy $V_i\in\U_i$.  
Each of the approximate solutions $V_{i+1}$ is obtained  from $V_i$ by an application of 
Lemma \ref{lemma-k-open} . 
This suggests in which sense the sets $\U_i$ have to approximate $K$.

\begin{definition}\label{in-approximation}
{\rm
Let $K\subset\Mmn$. We say that a sequence of open sets $\{\U_{i}\}\subset\Mmn$ 
is an {\em in-approximation} of $K$ if the following three conditions hold:

1. $\dsp \U_{i}\subset \U_{i+1}^{L}$\,;

2. the sets $\U_{i}$ are uniformly bounded;

3. if a sequence $F_{i}\in\U_i$ converges to $F$ as $i\to\infty$, then $F\in K$.
}
\end{definition}

\no
The name 'in-approximation' was introduced by Gromov \cite{gromov}.
Note that a necessary condition for $K$ to admit an in-approximation is that the set 
$\Int (K^L)$  is non-empty. Note also that the notion of in-approximation is related
to a notion of convexity. In this section we use lamination convexity with respect to
the cone of matrices of rank (at most) $n-1$ because Lemma \ref{basic}
only holds if $\rank(A-B) \le n-1$. In the next section we will prove a similar
lemma, but without any restriction. Thus in that section the natural cone is the whole
space (in that case $\R^{10}$) and in condition 1. in the in-approximation we will use the ordinary
convex hull.

\medskip


We are now ready to state the main result of this section.

\begin{theorem}\label{mainthm}
Assume that $K$ admits an in-approximation by open sets $\U_{i}$ and let $F\in \U_{1}$. 
Then, for each $\delta>0$, there exists $V_{\delta}\in L^{\infty}(\Om;\Mmn)$ such that   

\begin{align}
\label{con1}
& V_{\delta}=\Lcal (H_{\delta}) +F \text{ with } 
H_{\delta}\in \big(W^{1,\infty}(\Om;\Mnn)\big)^m \,,\\
\label{con2}
& V_{\delta}\in K \quad {\rm a.e.}\,, \\
\label{con3}
& \|H_{\delta}\|_{L^{\infty}(\Om)}<\delta   \,,\\
\label{con4}
& H_{\delta}|_{\partial \Om}=0\,.
\end{align}

\end{theorem}

\begin{proof}
We construct a sequence of piecewise constant divergence free maps $V_{i}$ such 
that 

\begin{align}\label{sequence-V_{i}}
& V_{i}=\Lcal (H_{i}) +F \text{ with } 
H_{i}\in \big(W^{1,\infty}(\Om;\Mnn)\big)^m \,,\\
\non & V_{i}\in \U_{i} \quad {\rm a.e.}\,, \\
\non & \|H_{i+1}-H_{i}\|_{L^{\infty}(\Om)}<\delta_{i+1}\,\\ 
\non & H_{i}|_{\partial \Om}=0\,.
\end{align}

\no
To start with, set $\dsp H_{1}:=0$ and $V_{1}:=F$. 
Since $F\in \U_{2}^{L}$, we can apply Lemma \ref{lemma-k-open} to deduce the existence of 
a function $V_{2}$ such that 

\begin{align*}
& V_{2}=\Lcal (G_{2}) +F \text{ with } G_{2}\in \big(W^{1,\infty}(\Om;\Mnn)\big)^m 
\text { and piecewise linear}\,,\\
& V_{2}\in \U_{2} \quad {\rm a.e.}\,, \\
& \|G_{2}\|_{L^{\infty}(\Om)}<\delta_{2}   \,,\\
& G_{2}|_{\partial \Om}=0\,.
\end{align*}

\no
with $\delta_{2}=\delta$. We then define $H_{2}=G_{2}$.
To construct $V_{i+1}$ and $\delta_{i+1}$ from $V_{i}$ and $\delta_{i}$, we proceed 
as follows. 
Let 

$$
\Om_{i}:=\{x\in\Om : \dist(x,\partial\Om)>1/2^{i})\}\,.
$$

\no
Let $\varro$ be a standard smooth convolution kernel
in $\R^{n}$, i.e., $\rho \geq 0$, $\int \rho = 1$, $\mbox{Spt} \, \rho \subset \{|x|<1\}$, 
and  let $\dsp\varro_{\e_{i}}(x):=\e_{i}^{-n}\varro(x/\e_{i})$. 
We choose $\dsp \e_{i}\in(0,2^{-i})$ so that 

\begin{equation}\label{est1}
\big\|\varro_{\e_{i}}*\Lcal (H_{i})-\Lcal (H_{i})\big\|_{L^{1}(\Om_{i})}<\frac{1}{2^{i}}\,.
\end{equation}

\no
where the convolution acts on each entry of the matrix field $\Lcal (H_{i})$.
Now let 

\begin{equation}\label{delta-choice}
\delta_{i+1}=\delta_{i}\e_{i}\,.
\end{equation}

\no
and use Lemma \ref{lemma-k-open} to construct a function 
$G_{i+1}\in \big(W^{1,\infty}(\Om;\Mnn)\big)^m$ such that 

\begin{align*}
& \Lcal (G_{i+1})+V_{i}\in \U_{i+1}\quad {\rm a.e.}  \,,\\
& \|G_{i+1}\|_{L^{\infty}(\Om)}<\delta_{i+1}\,.
\end{align*}

\no
Next we set $\dsp H_{i+1}:=\sum_{j=2}^{i+1}G_{j}$ and define $V_{i+1}$ 
according to \eqref{sequence-V_{i}}, so that  
$$
V_{i+1}=\Lcal (G_{i+1})+V_{i}\,.
$$

\no
Since $\dsp\sum_{i=2}^{\infty}{\delta_{i}}<\delta/2$ and, for $i>j$,

\begin{equation}\label{cauchy-est}
\|H_{i}-H_{j}\|_{L^{\infty}(\Om)} \leq\sum_{k=j+1}^{i}\|G_{k}\|_{L^{\infty}(\Om)}\,,
\end{equation}

\no
we find that $H_{i}\to H_{\infty}$ uniformly. 
Moreover, since by construction the sequence $\{H_{i}\}$ is 
uniformly bounded in $W^{1,\infty}(\Om)$, we have that 
$H_i \weakst H_\infty$ in $W^{1,\infty}$ weak*. In particular

$$
\Lcal (H_{i})\weakst \Lcal (H_{\infty}) \,\text{ in }L^{\infty} \text{ weak }*. 
$$

\no
Taking $\dsp H_{\delta}=H_{\infty}$ 
and $\dsp V_{\delta}:=\Lcal (H_{\delta})+F$, 
we see that conditions \eqref{con1}, \eqref{con3} and \eqref{con4} hold. 
We are left to show that $\dsp V_{\delta}\in K$ a.e.  
To this end, we will prove the strong convergence of $\Lcal (H_{i})$ to $\Lcal (H_\infty)$ 
in $L^{1}$. 
Indeed, since 
$$ 
\int_{\Om}\big(\Lcal (\Phi)(y)\big)_{kj}\varro(x-y)dy=-\sum_{i=1}^n\int_{\Om} 
\Phi^k_{ij}(y)\frac{\partial\varro}{\partial x_i}(x-y)dy\,,
\quad \forall\,\Phi\in \big(W^{1,\infty}(\Om;\Mnn)\big)^m\,,
$$ 
\no
and since $\|\nabla\varro_{\e_{i}}\|_{L^{1}}<C/\e_{i}$, we deduce from 
\eqref{delta-choice} and \eqref{cauchy-est}

\begin{align}\label{est2}
\non 
\big\|\varro_{\e_{i}}*\big(\Lcal (H_{i})-\Lcal (H_{\infty})\big)\big\|_{L^{1}(\Om_{i})}
&\leq\frac{C}{\e_{i}}\big\| H_{i}- H_{\infty}\big\|_{L^{\infty}(\Om)}\\
\non &\leq\frac{C}{\e_{i}} \sum_{k=i+1}^{\infty}\delta_{k}\\
\non &\leq 2\frac{C}{\e_{i}}\delta_{i+1}\\
&\leq C'\delta_{i}\,.
\end{align}

\no
Combining \eqref{est1} and \eqref{est2} we get

\begin{align*}
\big\|\Lcal (H_{i})-\Lcal (H_{\infty})\big\|_{L^{1}(\Om)}\leq C'\delta_{i}
+2^{-i}
&
+\big\|\varro_{\e_{i}}*\Lcal (H_{\infty})-
\Lcal (H_{\infty})\big\|_{L^{1}(\Om_{i})}\\
&+\big\|\Lcal (H_{i})-\Lcal (H_{\infty})\big\|_{L^{1}(\Om\setminus\Om_{i})}\,.
\end{align*}

\no
Since $\Lcal (H_{i})$ and $\Lcal (H_{\infty})$ are bounded, we obtain 
$\Lcal (H_{i})\to\Lcal (H_{\infty})$ in $L^{1}(\Om)$ and thus 
$V_{i}\to V_{\delta}$ in $L^{1}(\Om)$. 
Therefore there exists a subsequence $V_{i_{j}}$ such that 

$$
V_{i_{j}}\to V_{\delta} \quad{\rm a.e.}\,
$$

\no
It follows from the definition of in-approximation that 

$$
V_{\delta}\in K \quad {\rm a.e.}
$$

\end{proof}

\section{Applications of the convex integration results to the study of the Born-Infeld equations}
\label{young-meas}

\subsection{Approach by Young measures}\label{YM-subsec}

We formulate problem \eqref{weakpb1}-\eqref{weakpb2} in the language of $\Acal$-convexity 
(see, {\em e.g.}, \cite{fon-muller}, \cite{tartar}).
Let $\Omega\subset\R^{3}$ be an open bounded domain and let $\Mcal$ be defined by \eqref{BI-manifold}. 
Let $A^{(1)},A^{(2)},A^{(3)}\in\Mtten$ be defined as follows

\begin{align*}
& A^{(1)}=\left(
\begin{array}{llllllllll}
1 & 0 & 0 & 0 & 0 & 0 & 0 & 0 & 0 & 0 \\
0 & 0 & 0 & 1 & 0 & 0 & 0 & 0 & 0 & 0
\end{array}
\right)\,, \\
& A^{(2)}=\left(
\begin{array}{llllllllll}
0 & 1 & 0 & 0 & 0 & 0 & 0 & 0 & 0 & 0 \\
0 & 0 & 0 & 0 & 1 & 0 & 0 & 0 & 0 & 0
\end{array}
\right)\,, \\  
& A^{(3)}=\left(
\begin{array}{llllllllll}
0 & 0 & 1 & 0 & 0 & 0 & 0 & 0 & 0 & 0 \\
0 & 0 & 0 & 0 & 0 & 1 & 0 & 0 & 0 & 0
\end{array}
\right)\,.
\end{align*}

\no
We introduce the operators 

\begin{align*}
& \Acal(V):=\sum_{i=1}^{3}A^{(i)}\frac{\partial V}{\partial x_{i}}\,, \quad 
V:\Om\to \R^{10}\,,\\
& \A(w):=\sum_{i=1}^{3}A^{(i)}w_{i}\in\Lin(\R^{10};\R^{2})\,, \quad w\in\R^{3}\,,
\end{align*}

\no
where $\Lin(\R^{10};\R^{2})$ denotes the space of linear operators from $\R^{10}$ to 
$\R^{2}$. 
The operator $\Acal$ satisfies the {\em constant rank} property, {\em i.e.}, 

\begin{equation*}
\rank\A(w)=2 \quad \forall\, w\in\S^{2}\,,
\end{equation*}

\no
where $\S^{2}$ is the unit sphere in $\R^{3}$. 
Moreover 

\begin{equation*}
\ker\A(w)=\{(\alpha,\beta,\gamma)\in\R^{3}\times\R^{3}\times\R^{4}: 
\alpha\perp w\,,\beta\perp w\}=\R^{2}\times\R^{2}\times\R^{4}\,.
\end{equation*}

\no
Therefore the {\em characteristic cone} $\Lambda$ is all of $\R^{10}$. Indeed 

\begin{equation*}
\Lambda:=\cup _{w\in S^{2}}\ker\A(w)=\{(\alpha,\beta)\in\R^{3}\times\R^{3}:
\exists\,\xi\in\R^{3} \text{ \rm such that }\xi\perp\alpha,\xi\perp\beta \}\times\R^{4}=
\R^{3}\times\R^{3}\times\R^{4}\,.
\end{equation*}

\no
Thus $\Lambda$-convexity reduces to standard convexity. 
In terms of 
the constant rank operator $\Acal$ our problem reads as 

\begin{align}
\label{problemA}& \Acal (V_{j})=0  \quad\text{in }\dpr(\Om),\\
\label{inclusion}& V_{j}\in \Mcal \quad\text{a.e. in } \Om\,. 
\end{align}

\no
One can also consider the approximate version of \eqref{problemA}, where 
the differential constraint on the sequence $\{V_{j}\}$ is replaced by the weaker condition 

\begin{equation}\label{approx}
\Acal (V_{j})\to 0  \quad\text{strongly in }W^{-1,p'}(\Om)\,.
\end{equation}

\no
The next Theorems \ref{YM1} and \ref{converse} and their corollaries are special case 
of more general results contained in \cite{fon-muller}, where more general constant-rank operators 
are considered. Let us also mention that, in the gradient case, i.e., when the operator $\Acal$ is the 
curl operator, such results were first established by Kinderlehrer and Pedregal \cite{kind-ped}.

\begin{theorem}\label{YM1}
Let $1\leq p<+\infty$. 
Suppose that the sequence $\{V_{j}\}$ generates the Young measure  $\{\nu_x\}_{x\in\Om}$
and let $V_{j}\weak V$ in $L^{p}(\Om;\R^{10})$. 
If $\{V_{j}\}$ 
satisfies \eqref{problemA}, or its approximate 
version \eqref{approx}, then 
\begin{align*}
& \langle\nu_{x},id\rangle=V(x)\in \ker \Acal \,,\\
& \int_{\Om}\int_{\R^{10}}|M|^p d\nu_x(M) < \infty\,.
\end{align*}
\no
If in addition the sequence $\{V_{j}\}$ is uniformly bounded in 
$L^{\infty}(\Om;\R^{10})$ and \eqref{inclusion} holds, 
then 

\begin{equation}\label{suppYM}
\supp\,\nu_{x}\subset \Mcal \text{ for }\text{\rm a.e. } x\in\Om\,.
\end{equation}
\end{theorem}

\begin{corollary}
Under the assumptions of Theorem \ref{YM1}, if \eqref{suppYM} holds, then 
\begin{equation*}
V(x)\in \Mcal^{c} \text{ for } \text{ \rm a.e. }x\in\Om\,.
\end{equation*}
\end{corollary}

\begin{theorem}\label{converse}
Let $1\leq p<+\infty$, and let $\{\nu_x\}_{x\in\Om}$ be a weakly measurable family of 
probability measures on $\R^{10}$. Suppose that 

\begin{align*}
& \langle\nu_{x},id\rangle\in \ker \Acal \,,\\
& \int_{\Om}\int_{\R^{10}}|M|^p d\nu_x(M) < \infty\,.
\end{align*}
\no
Then there exists a sequence $\{V_j\}\subset L^{p}(\Om;\R^{10})$ satisfying 
\eqref{problemA} that generates $\{\nu_x\}$.

\end{theorem}

\begin{corollary}\label{coroll2}
Let $V\in L^{p}(\Om;\R^{10})$. Suppose that $\Acal (V)=0$ and $V\in \Mcal^{c}$ a.e. 
Then there exists a sequence $\{V_{j}\}\subset L^{p}(\Om;\R^{10})$ satisfying 
\eqref{problemA} such that
\begin{equation*}
\dist(V_{j},\Mcal)\to 0 \text{ in } L^{p}(\Om) \text{ and } 
V_{j}\weak V \text{ in } L^{p}(\Om;\R^{10}) \,.
\end{equation*}  
\end{corollary}

\begin{remark}
{\rm
By suitably projecting the sequence $\{V_j\}$ 
provided by Corollary \ref{coroll2} onto $\Mcal$, one can obtain 
a sequence $\{\tilde V_{j}\}\subset L^{p}(\Om;\R^{10})$ satisfying \eqref{approx} 
such that 

\begin{align*}
& \tilde V_{j}\in \Mcal \text{ \rm a.e. } \text{ and }\,\, 
\tilde V_{j}\weak V \text{ in } L^{p}(\Om) \,.
\end{align*}
}
\end{remark}

\subsection{Approach by convex integration}\label{section-A}

\no
We now use the convex integration approach developed in Section \ref{conv-integ-section} 
to find maps which satisfy  the constraints \eqref{problemA} and \eqref{inclusion} exactly 
and have a prescribed average in the interior of the  convex hull $\Mcal^{c}$. Then Theorem   \ref{final-thm}
will follow easily by partitioning $\Omega$ into small subdomains and applying the result
to each subdomain. 


As above we write
\begin{equation*}
M=(M_{D},M_{B},M_{P},M_{h})\in
\R^{3}_{D}\times\R^{3}_{B}\times\R^{3}_{P}\times\R_{h}\,.
\end{equation*}
We look for maps 
\begin{equation*}
V :\Omega \subset \R^3 \to \R^{10} 
\end{equation*}
which satisfy the constraints $\div V_D = \div V_B =0$. 
Since we take the divergence of a matrix rowwise this constraint can be written in the
compact form
\begin{equation*}
\div  \binom{V_D^T}{V_P^T} = 0   \qquad \mbox{where   }   \binom{V_D^T}{V_B^T}  \in \Mtwo.
\end{equation*}
We first state the counterpart of Lemma \ref{basic} in the present setting. 
The operator $\Lcal$ is the same as in the previous section. As we work with $n=3$
we could identify $\Lcal$ with the rowwise curl operator of a matrix, but we refrain 
from doing so to keep the notation as close as possible to the previous section. 

\begin{lemma}\label{basic-A}
Let $M,N\in \R^{10}$ and let $F:=\t M+(1-\t)N$ for some $\t\in(0,1)$. 
Then for each $\delta>0$, there exists $V\in L^{\infty}(\Om;\R^{10})$ such that 

\begin{align}
\label{th1}&    \binom{V_D^T}{V_B^T}  =   \binom{F_D^T}{F_B^T}  + \Lcal(G)
    \text{ with } G \in (W^{1,\infty}(\Om;\Mtts))^2 
\text { and piecewise linear} \,,\\
\label{th2}& \|G\|_{L^{\infty}(\Om)}<\delta \,,  \\
   \label{th3}& G|_{\partial \Om}=0\, , \\
\label{th4}&\dist\big(V,\{M, N\}\big)
                  <\delta \,,\\
\label{th6}& \int_{\Om} V\, dx=F  |\Omega|\,.
\end{align}
\end{lemma}

\begin{proof}

By scaling we may assume without loss of generality $|\Omega| = 1$. 
We apply Lemma \ref{basic} with
\begin{equation}
A =  \binom{M_D^T}{M_B^T}, \quad B =  \binom{N_D^T}{N_B^T}
\end{equation}
and with $\delta'$ instead of $\delta$. Note that $A, B \in \Mtwo$ and hence
\begin{equation*}
\rank (A-B) \le 2 = n-1.
\end{equation*}
It follows that there exists a piecewise linear $G \in (W^{1,\infty}_0(\Om;\Mtts))^2$
such that \eqref{th1} and \eqref{th3} hold and
\begin{equation}  \label{eq:deltaprime}
\| G \|_{L^\infty(\Omega)} < \delta', \qquad \dist\left( \binom{V_D^T}{V_B^T}, \{A, B\} \right) < \delta'.
\end{equation}
Moreover
\begin{equation*}
\int_\Omega (V_D, V_B) = (F_D, F_B)
\end{equation*}
since $G= 0$ on $\partial \Omega$. 

It remains only to define $V_P$ and $V_h$. Since no differential constraint is imposed on these variables
this is easy. We set
\begin{equation*}
\Omega_A :=\left\{ x \in \Omega :  \left|   \binom{V_D^T}{V_B^T} - A \right| < \left|  \binom{V_D^T}{V_B^T} - B \right|  \right\}.
\end{equation*}
Denote by  $\chi_A$ be the characteristic function of $\Omega_A$ and define
\begin{equation*}
\eta := \int_\Omega \chi_A = |\Omega_A| \, .
\end{equation*}
We set
\begin{equation*}
(V_P, V_h) = \chi_A  (M_P, M_h) +  (1 - \chi_A) (N_P, N_h) +  (\theta - \eta)  (M_P- N_P, M_h - N_h).
\end{equation*}
Then  by the definition of $\eta$ (recall that $|\Omega| = 1$)
\begin{equation*}
\int_\Omega (V_P, V_h) \, dx = (F_P, F_h).
\end{equation*}
Moreover the definition of $\Omega_A$ and the second inequality in 
\eqref{eq:deltaprime} imply that
 \begin{equation} 
\label{eq:distV}
\dist( V, \{(M, N)\}) \le \delta' + |\theta - \eta|  \, \,  |(M_P- N_P, M_h - N_h)|.
\end{equation}
To estimate $\eta - \theta$ we note that
\begin{align*}
&  \binom{F_D^T}{F_B^T} -  \eta A - (1 -\eta) B = 
\int_\Omega   \binom{V_D^T}{V_B^T}  \, dx  -   \eta A - (1 -\eta) B  \non \\
 = & \int_{\Omega_A}  \binom{V_D^T}{V_B^T} - A  \, dx  + 
 \int_{\Omega \setminus \Omega_A}  \binom{V_D^T}{V_B^T} - B \, dx .
\end{align*}
Taking the norm on both sides and using  the definition of $\Omega_A$ and \eqref{eq:deltaprime} we see that
$|  (\theta - \eta) (A-B) |  \le \delta' $.
Now take 
\begin{equation*}
\delta' := \frac12 \delta \,  \frac{|A-B|}{|M-N|} \le \frac12 \delta.
\end{equation*}
Then  \eqref{th4} follows from \eqref{eq:distV}.
\end{proof}

Now we introduce the appropriate definition of in-approximation for the 
Born-Infeld set $\Mcal$ defined in \eqref{BI-manifold}.  
Note that while in Lemma \ref{basic} we had the constraint $\rank(A-B) \le n-1$, 
in Lemma \ref{basic-A} there is no constraint at all on the matrices $M$ and $N$. 
Thus the lamination convex hull introduced in the previous section is replaced by  the ordinary convex hull
and the in-approximation is defined using the convex hull. Note that by Caratheodory's theorem
the convex hull of any set $E \subset \R^{10}$ satisfies $E^c = \bigcup_{i=0}^{10} E^{i}$, 
where $E^0 = A$ and $E^{i+1}$ is inductively defined as the set  all convex combinations $\theta A + (1- \theta) B$,
with $A, B \in E^{i}$.  

\begin{definition}\label{A-inapproximation}
We say that a sequence of open sets 
$\{\U_{i}\}\subset\R^{10}$ 
is an in-approximation of $\Mcal$ if the following three conditions hold:

1. $\dsp \U_{i}\subset \U_{i+1}^{c}$\,;

2. the sets $\U_{i}$ are uniformly bounded;

3. if a sequence $F_i$ converges to $F$ as $i \to \infty$ and  $F_i \in \U_i$ for each $i$, then $F\in \Mcal$.

\end{definition}

Regarding the existence of  in-approximations with respect to ordinary convexity we have the following abstract result.

\begin{lemma}\label{covering} Let $M \subset \R^d$, assume that $\Int M^c \ne \emptyset$
and let $L\subset \Int \Mcal^c$ be compact. Then there exist $R> 0$ and open sets $\U_i$
such that 
\begin{enumerate}
\item[1.] $L \subset \U_1$
\item[2.] $\U_i \subset \U_{i+1}^c  \quad \forall \, i \ge 1$,
\item[3.] $\U_i \subset  B(0,R)   \quad \forall \, i \ge 1$,
\item[4.]  \label{item:inapp}  $\U_i  \subset  B_{1/i}(M)  \quad \forall \, i \ge 2$.
\end{enumerate}
In particular the sets $\U_i$ are an in-approximation of $M$. 
\end{lemma}

Note that $\U_i  \subset  B_{1/i}(M)$ if and only if  $\dist(p, M) < \frac1i$ for all $p \in \U_i$.

To prove Lemma \ref{covering} we will inductively use the following elementary result for finite sets.

\begin{lemma} \label{lem:finiteset} Let $E \subset \R^d$ be a finite set,  let $F$ be a finite set
with 
\begin{equation} 
F \subset \Int E^c, 
\end{equation}
and let $\varepsilon > 0$.  Then there exists a finite set $F'$ such that
\begin{equation}
F \subset \Int (F')^c, \quad  F' \subset \Int E^c, \quad F' \subset B_\varepsilon(E).
\end{equation}
\end{lemma}

\begin{proof} Step 1. Assume that $F = \{0\}$. \\
By assumption there exists $\eta > 0$ such that
\begin{equation*}
B_\eta(0) \subset E^c.
\end{equation*}
Let $\delta \in (0,1)$ and set $F' = (1-\delta) E$. 
Then $(F')^c \supset B_{(1- \delta) \eta}(0)$ and hence
$ 0 \in \Int (F')^c$. Since $E^c$ is convex we also have
for every $p \in E$ the inclusion $\delta B_\eta(0) + (1-\delta) p \subset E^c$.
Thus $F' \subset \Int E^c$.

Finally if $\delta < \varepsilon/ \max \{ |p| : p \in E \}$ we have $F' \subset B_\varepsilon(E)$.

Step 2. General finite $F$.\\
Let
\begin{equation*}
\delta < \frac{\varepsilon}{\max \{ |q - p| : q \in F, p \in E \}}.
\end{equation*}
For $q \in F$ define
\begin{equation*}
F'_q :=  q + (1-\delta) (-q + E).
\end{equation*}
By Step 1
\begin{equation*}
q \in \Int (F'_q)^c, \quad  F'_q \subset \Int E^c, \quad F'_q \subset B_\varepsilon(E).
\end{equation*}
Thus $F' := \bigcup_{q \in F} F'_q$ has the desired properties.
\end{proof}

\begin{proof}[Proof of Lemma \ref{covering}]
For each $q \in L$ there exists a $\delta(q) > 0$ such that
the cube $q + (-3 \delta, 3 \delta)^d$ is contained in $M^c$.
Since $L$ is compact there exist $q_1, \ldots, q_m \in L$ such that
\begin{equation*}
L \subset  \bigcup_{i=1}^m  q_i + (-\delta_i, \delta_i)^d , \quad  \mbox{and}  \quad
q_i + (-3\delta_i, 3\delta_i)^d\subset M^c.
\end{equation*}
Let $F_1$ be the set of all the corner points of all the cubes
$q_i + [- \delta_i,  \delta_i]^d$. Then
\begin{equation}  \label{eq:F1K}
L \subset F_1^c.
\end{equation}
We now show that there exists a finite set $E \subset M$ such that
$F_1 \subset \Int E^c$.  Indeed, let $G_1$ denote the set of all
the corner points of  the cubes $q_i + [- 2 \delta_i, 2 \delta_i]^d$.
By Caratheodory's theorem each point in $G_1$ is a convex combination 
of at most $d+1$ points in $M$. Thus there exists a finite set $E \subset M$
such that $G_1 \subset E^c$. 
This implies that
\begin{equation*}
F_1 \subset \Int G_1^c \subset \Int E^c.
\end{equation*}
Set
\begin{equation*}
R := \max \{  |p| : p \in E \}.
\end{equation*}

Inductive application of Lemma \ref{lem:finiteset}
yields finite sets $F_i$ with $F_i \subset \Int E^c$ for all $i \ge 1$ and
\begin{equation}  \label{eq:in_finite}
F_i  \subset \Int F_{i+1}^c \quad \forall i \ge  1, \qquad F_i \subset B_{\frac1i}(E) \quad \forall i \ge 2. 
\end{equation}
Moreover the condition $F_i \subset \Int E^c$ implies that
\begin{equation*}
F_i \subset B(0,R) \quad \forall i \ge 1.
\end{equation*}
Now define
\begin{align*}
\U_1 &:= \Int F_{2}^c, \\
\U_i &:= \Int F_{i+1}^c \cap B_{\frac1i}(E) \quad \forall i \ge 2.
\end{align*}
Then the sets $\U_i$ are open and $\U_i \subset B(0,R)$ since 
$F_{i+1} \subset B(0,R)$.
Moreover by \eqref{eq:in_finite} we have $F_i \subset \U_i$ for all $i \ge 1$ 
and thus 
\begin{equation*}
\U_i   \subset F_{i+1}^c \subset \U_{i+1}^c.
\end{equation*}

Since $\U_1$ is convex the inclusion  $F_1 \subset \U_1$ and \eqref{eq:F1K}
imply that $ \U_1 \supset F_1^c \supset L$.
This finishes the proof of Lemma \ref{covering}.
\end{proof}

\bigskip 
%
%
%
%

For future reference we also note the following observation.

\begin{lemma} \label{lem:convergencePh}
For every $\delta > 0$ and every $R > 0$ there exists an $\eta > 0$
such that the estimates
\begin{equation}
|(D,B,P,h)| \le R \quad \mbox{and} \quad \dist((D,B,P,h), \Mcal) < \eta
\end{equation}
imply that
\begin{equation}
|P - D \wedge B| < \delta \quad \mbox{and} \quad   |h - \sqrt{1 + |D|^2 + |B|^2 + |P|^2} | < \delta.
\end{equation}
\end{lemma}

\begin{proof} This just follows from the continuity of the functions involved and compactness. 
Indeed, if the assertion is false then there exist $R_0 > 0$ and $\delta_0 > 0$ such that
for each $k \in \N$, $k\ge 1$ there exist $(D_k, B_k, P_k, h_k) \in \bar B(0, R_0)$ such that
\begin{equation*} 
\dist( (D_k, B_k, P_k, h_k) , \Mcal) \le \frac1k
\end{equation*}
and 
\begin{equation} \label{eq:convPh}
|P_k - D_k \wedge B_k| +   |h_k - \sqrt{1 + |D_k|^2 + |B_k|^2 + |P_k|^2} | \ge \delta_0.
\end{equation}
There exists a subsequence such that 
$(D_{k_j}, B_{k_j}, P_{k_j}, h_{k_j})  \to (D,B,P,h)$ and $(D,B,P,h) \in \Mcal$.
Passage to the limit in \eqref{eq:convPh}  along this subsequence yields
\begin{equation}
|P - D \wedge P| +  |h - \sqrt{1 + |D|^2 + |B|^2 + |P|^2} | \ge \delta_0,
\end{equation}
but this contradicts the definition of $\Mcal$. 
\end{proof}


We can now construct solutions of the problem 
$V \in \Mcal$ and $\div V_D = \div V_B = 0$
in complete analogy with the argument in the previous section.

\begin{lemma}\label{lemma-km}
Let $U \subset \R^d$ be open and bounded. 
Suppose that
$F\in L^{\infty}(\Om;\R^{10})$  is  a piecewise constant function which satisfies 
\begin{align*}
& 
\div F_D = \div F_B = 0 \text { in }\dpr(\Om)\,,\\
& F\in U^c \,\text{ \rm a.e. } 
\end{align*}

\no 
Then, for each $\delta>0$, there exists  $\Vd\in L^{\infty}(\Om;\R^{10})$ such that 
\begin{align}
& \binom{(\Vd)_D^T}{(\Vd)_B^T} = F + \Lcal(G) 
\text{  with } G \in (W^{1, \infty}(\Om; \Mtts))^2
\text { and piecewise linear} \,,\\
& \|G\|_{L^{\infty}(\Om)}<\delta \,   , \\
& G|_{\partial \Om}=0\,, 
 \\
& \Vd\in U\quad {\rm a.e.}\,, \\
& \int_{\Om}\Vd\,dx=  \int_{\Om} F \, dx \,.
\end{align}
\end{lemma}

%
%
%

\begin{proof} 
Lemma \ref{lemma-km} follows by induction from Lemma \ref{basic-A} exactly in the same
way as  Lemma \ref{lemma-k-open} was deduced  from Lemma \ref{basic}.
\end{proof}


\begin{theorem}\label{thm99}
Let $L$ be a compact subset of $\Int\Mcal^c$. Then there exists an $R > 0$ such that for all $F \in L$
there exists $V \in L^{\infty}(\Om;\R^{10})$
such that 
\begin{align}
&  \label{eq:finaldiv} \binom{V_D^T}{V_B^T} =  \binom{F_D^T}{F_B^T} + \Lcal(H) 
\text{  with } H \in (W^{1, \infty}(\Om; \Mtts))^2\\
& \|H\|_{L^{\infty}(\Om)}<  1   
 \,, \\
& H_{|\partial \Om}=0\,, \\
& V \in \Mcal \text{ {\rm a.e.}}  \quad \mbox{and}   \quad  \|V \|_{L^\infty} \le R \,,\\
& \int_{\Om} V  \,dx=F\,.
\end{align}
\end{theorem}


\begin{proof}
By Lemma \ref{covering} there exists  an in-approximation $\U_i$ with 
$L \subset \U_1$.
Arguing exactly as  in the proof of  Theorem \ref{mainthm} and using now
Lemma \ref{lemma-km} instead of Lemma \ref{lemma-k-open} we can inductively define $\delta_i$, $H_i$ and $\e_i$ such that
\begin{align*}
&\binom{(V_i)_D^T}{(V_i)_B^T}  = \binom{F_D^T}{F_B^T}  + \Lcal(H_i)  \\
&H_i \in (W^{1, \infty}(\Omega, \Mtts))^2 \quad \text{and piecewise linear }\\
&H_i{|\partial \Omega = 0} \\
& V_i \in \U_i \text{  a.e.}\\
&\int_\Om  V_i \, dx = F \, |\Omega|\\
&\| \rho_{\e_i} \ast \Lcal(H_i)  - \Lcal(H_i) \|_{L^1(\Omega_i)}  < \frac{1}{2^i}\\
&\delta_{i+1} = \delta_i \e_i \\
&\| H_{i+1} - H_i \|_{L^\infty(\Om)} < \delta_{i+1}.
\end{align*}
Then we get   $\Lcal(H_i) \weakst \Lcal(H)$ in $L^\infty(\Om)$, $\|H\| < \delta \le 1$ and 
$\Lcal(H_i) \to  \Lcal(H)$  in $L^1(\Om)$. This implies that
\begin{equation}
(V_i)_D \to V_D,  \quad (V_i)_B \to V_B \qquad \text{in $L^p(\Omega)$ for all $p < \infty$}.
\end{equation}
%
%
%
%
%
%
Thus
\begin{equation*}
(V_{i})_{D}\wedge (V_{i})_{B}
\to V_{D}\wedge V_{B} \,
\text{ \rm{ in}  $L^{p}(\Om)$ for all $p < \infty$}  \,.
\end{equation*}


%
\no
By the construction of the in-approximation we have
$\| \dist(V_{i},  \Mcal) \|_{L^\infty} \le \frac1i$ and $\|V_i\|_{L^\infty} \le R$ (where 
$R$ depends only on $L$). Thus
Lemma \ref{lem:convergencePh} implies that
\begin{equation*}
\| (V_{i})_{P} -  (V_{i})_{D}\wedge (V_{i})_{B}  \|_{L^\infty} \to 0
\end{equation*}
and thus
\begin{equation*}
(V_{i})_{P}\to V_{D}\wedge V_{B} \,
\text{ \rm{ in} } L^{p}(\Om)\,.
\end{equation*}
for all $p < \infty$. 
\no
Similarly Lemma \ref{lem:convergencePh} implies that 
\begin{equation*}
\Big\|(V_{i})_{h} - 
\sqrt{1+|(V_{i})_{D}|^{2}+|(V_{i})_{B}|^{2}+|(V_{i})_{P}|^{2}}
\Big\|_{L^{\infty}(\Om)} \to 0
\end{equation*}
%
%
\no
and therefore $\dsp (V_{i})_{h}\to V_{h}$ strongly in $L^{p}(\Om)$
and $V \in \Mcal$ a.e.  Since $|V_i| \le R$ a.e it follows also that $|V| \le R$ a.e.
\end{proof}

\begin{proof}[Proof of Theorem  \ref{final-thm}]
Let $L \subset \Int\Mcal^c$ be compact and suppose that $F \in L^\infty(\Om,\R^{10})$ is piecewise constant
with $F(x) \in L$ a.e. 
Assume furthermore that $\div B = \div D = 0$ in the sense of distributions where $F=(B,D,P,h)$. 

To construct the approximation $V_j$ we may assume that 
\begin{equation}
\mbox{diam}\, \Omega_i \le  \frac{1}{j}
\end{equation}
for all $i$
since otherwise we can always subdivide all the sets $\Omega_i$ with larger diameter until this condition is satisfied.

Now we apply Theorem \ref{thm99} to $\Omega_i$ and $F_i$ and we obtain a function $V^j_i: \Omega_i \to \R^{10}$ and a potential
$H^j_i : \Omega_i \to (\Mtts)^2$.
We extend $H^j_i$ and $V^j_i - F_i$ by zero outside $\Omega_i$. Using again \eqref{eq:extension}
we see that these extensions satisfy
\begin{equation*}
\chi_{\Omega_i} \left(\binom{(V^j_i)_D^T}{(V^j_i)_B^T}  - \binom{F_D^T}{F_B^T}  \right) = \Lcal(\chi_{\Omega_i} H_i^j )  
\end{equation*}
 and thus  $\div  \chi_{\Omega_i}  (V^j_i - F_i)_B = \div  \chi_{\Omega_i} (V^j_i - F_i)_D = 0$
in the sense of distributions in $\R^3$. Finally we set
\begin{equation}
V^j =  F + \sum_i  \chi_{\Omega_i} (V^j_i - F_i).
\end{equation}
Then $\div B^j = \div D^j = 0$, where $V^j = (D^j, B^j, P^j, h^j)$.  Moreover $V^j = V_i^j$ in $\Om_i$ and hence $V^j \in \Mcal$ a.e.

It remains to show that $V^j \weakst F$ in $L^\infty(\Omega)$. First consider Lipschitz continuous
test functions $\varphi$ and let $x_i$ be a point in $\Omega_i$. Then
\begin{align}
 \int_\Omega  (V^j - F)\varphi  \, dx = \sum_i   \int_{\Omega_i}    (V^j_i - F_i) \varphi \, dx
= \sum_i  \int_{\Omega_i}  (V^j_i (x)-F_i)   (\varphi(x) - \varphi(x_i)) \, dx,
\end{align}
where we used that $V^j_i - F_i$ has zero average in $\Omega_i$. Now
$|V^j_i - F_i| \le 2R$ and $|\varphi(x) - \varphi(x_i)| \le \mbox{Lip} \, \varphi  \, \, \mbox{diam}\, \Omega_i
\le \frac1j \mbox{Lip} \, \varphi   $.
Thus
\begin{equation}  \label{eq:weakcon}
\int_\Omega (V^j - F) \varphi \, dx \to 0
\end{equation}
for all Lipschitz continuous $\varphi$. Since these functions are dense in $L^1$ and $\| V^j - F \|_{L^\infty} \le 2R$
the convergence \eqref{eq:weakcon} holds for all $\varphi \in L^1$. This finishes the proof.
%
%
%
%
%
\end{proof}


\medskip
\centerline{\sc Acknowledgements}

We thank Yann Brenier for bringing our attention to this problem.
This work was largely  carried out when MP was post-doc at SISSA (Trieste).

\vspace{2mm}


\end{document}